\documentclass[a4,11pt]{article}
\usepackage{amsthm}

\usepackage[a4paper]{geometry}
\newcommand{\newnumbered}[2]{\newtheorem{#1}[theorem]{#2}}
\newcommand{\newunnumbered}[2]{\newtheorem{#1}[theorem]{#2}}
\newcommand{\Title}[2]{\title{#1}\newcommand{\Acknowledgements}{\section*{Acknowledgements} #2}}
\newcommand{\Author}[2][]{\author{#2}}

\newcommand{\br}{, }
\newcommand{\fs}{. }
\newcommand{\thanksone}[3][]{#1\thanks{#3\email{\tt #2}}}

\newcommand{\email}[1]{#1}
\newcommand{\classno}[2][2000]{}
\newcommand{\printscl}{}
\newcommand{\mktitle}{\maketitle}
\newcommand{\mkabstitle}{}

\usepackage[utf8x]{inputenc}
\usepackage{amsfonts}
\usepackage{amsmath}
\usepackage{amssymb}
\usepackage[mathscr]{eucal}
\usepackage{color}
\usepackage[colorlinks=true,citecolor=black,linkcolor=black,urlcolor=blue]{hyperref}
\usepackage{comment}
\newtheorem{theorem}{Theorem}[section] 
\newtheorem{lemma}[theorem]{Lemma}     

\newnumbered{assertion}{Assertion}    
\newnumbered{conjecture}{Conjecture}  
\newnumbered{definition}{Definition}
\newnumbered{hypothesis}{Hypothesis}
\newnumbered{remark}{Remark}
\newnumbered{note}{Note}
\newnumbered{observation}{Observation}
\newnumbered{problem}{Problem}
\newnumbered{question}{Question}
\newnumbered{algorithm}{Algorithm}
\newnumbered{example}{Example}
\newnumbered{res}{Result}
\newunnumbered{notation}{Notation} 
\numberwithin{equation}{section}

\renewcommand{\labelenumi}{\theenumi}

\newcommand{\noop}[1]{}

\DeclareMathOperator{\PG}{PG}
\DeclareMathOperator{\Q}{\mathsf{Q}^+}
\DeclareMathOperator{\T}{\mathcal{T}}

\begin{document}

\Title{On tight sets of hyperbolic quadrics}{%
The author is supported by
Basic Science Research Program through the National Research Foundation of Korea (NRF) funded
by the Ministry of Education (grant number NRF-2018R1D1A1B07047427).
}

\Author[Alexander L. Gavrilyuk]{%
\thanksone[Alexander L. Gavrilyuk]{alexander.gavriliouk@gmail.com}{%
Pusan National University\br
2, Busandaehak-ro 63beon-gil\br
Geumjeong-gu, Busan, 46241\br
Republic of Korea\fs
}
}

\classno{05E30 (primary), 05B15 (secondary)}

\date{\today}

\mktitle

\begin{abstract}
We prove that the parameter $x$ of a tight set $\mathcal{T}$ of a hyperbolic quadric $\Q(2n+1,q)$
of an odd rank $n+1$ satisfies ${x\choose 2}+w(w-x)\equiv 0\mod q+1$, where $w$ is the number of points
of $\mathcal{T}$ in any generator of $\Q(2n+1,q)$.
As this modular equation should have an integer solution in $w$ if such a $\mathcal{T}$ exists, 
this condition rules out roughly at least one half of all possible parameters $x$. 
It generalizes a previous result by the author and K. Metsch
shown for tight sets of a hyperbolic quadric $\Q(5,q)$ 
(also known as Cameron-Liebler line classes in $\PG(3,q)$).
\printscl
\end{abstract}

\mkabstitle

\section{Introduction}
Let $\PG(n,q)$ denote the projective space of dimension $n$
with underlying vector space $V:=\mathbb{F}_q^{n+1}$ over the finite field $\mathbb{F}_q$
with $q$ elements. 
For a non-degenerate quadratic (or reflexive sesquilinear) form $f$ on $V$,
the {\bf classical polar space} $\mathsf{P}$ associated with $f$ is the incidence structure
formed by the totally singular (or totally isotropic, respectively) subspaces of $f$
and their incidence is defined by symmetrized containment \cite{pps}.
We consider the elements of $\mathsf{P}$ as subspaces of $\PG(n,q)$,
so they are projective points, lines, $\ldots$.
A maximal subspace of $\mathsf{P}$ has dimension $r-1$, where $r$ is the Witt index of $f$, 
also called the {\bf rank} of $\mathsf{P}$, and such a subspace is called a {\bf generator}.

We will consider the polar space $\mathsf{Q}^+(2n+1,q)$ of rank $n+1$
defined by a hyperbolic quadric $f$, i.e., the set of projective points of $\PG(2n+1,q)$
satisfying $f({\bf x})=0$, where
\begin{equation*}
f({\bf x}):=f(x_0,\ldots,x_{2n+1})=x_0x_1 + \cdots + x_{2n}x_{2n+1},~{\bf x}\in V.
\end{equation*}

The associated bilinear form $b({\bf x},{\bf y}):=f({\bf x}+{\bf y})-f({\bf x})-f({\bf y})$
defines the {\bf polarity} $\perp$ of $\PG(2n+1,q)$.
Two points of the polar space are collinear if and only if $b({\bf x},{\bf y})=0$.
For a point $P$ of the quadric, denote by $P^{\perp}$ the set of points collinear with $P$,
which form the (tangent) hyperplane of $P$.
Note that $P\in P^{\perp}$, and, for a point set (or a subspace) $S$, let $S^{\perp}$ denote $\cap_{P\in S}P^{\perp}$.

The notion of tight sets was introduced by Payne \cite{Payne} for generalized quadrangles
(which include the classical polar spaces of rank $2$), and it was extended
to polar spaces of higher rank by Drudge \cite{D}.
Tight sets are extremal sets of points in the following sense.
It was shown \cite[Theorem~8.1]{D}
that, for a set $\mathcal{T}$ of points of a finite polar space $\mathsf{P}$ of rank $r$
over $\mathbb{F}_q$, the average number $\kappa$ of points of $\mathcal{T}$ collinear to a given
point of $\mathcal{T}$ is bounded above by
\begin{equation}\label{eq-tight}
\kappa\leq |\mathcal{T}|\frac{q^{r-1}-1}{q^r-1}+q^{r-1},
\end{equation}
and if equality attains, then $\mathcal{T}$ is said to be {\bf tight}.
Moreover, in this case $|\mathcal{T}|=x\frac{q^{r}-1}{q-1}$ holds for some non-negative integer $x$,
which is called the {\bf parameter} of the tight set $\mathcal{T}$.
The complement of an $x$-tight set is a $(q^{r-1}+1-x)$-tight set,
so that we may assume $x\leq \frac{q^{r-1}+1}{2}$.

The question is to determine for which parameters $x$ an $x$-tight set exists,
and to classify the examples admitting a given parameter $x$.
The answer fundamentally depends on the type and the rank of $\mathsf{P}$,
see \cite{tightpolar,BLP,dBM,M16,MW} for some recent results.
Besides that, geometric properties and characterizations of tight sets
are of interest, as they nicely interact with related structures
of polar spaces such as $m$-ovoids. This fact can be explained
from the point of view of algebraic graph theory \cite{tightpolar}, regardless of
the type of a polar space.

Note that a disjoint union of $x$ generators of $\mathsf{P}$, which are the only tight sets
with parameter $1$ \cite[Theorem~9.1]{D}, is itself a tight set with parameter $x$.
As for the hyperbolic polar spaces $\mathsf{Q}^+(2n+1,q)$, their tight sets in the partial case $n=2$
appeared in a different context in a paper by Cameron and Liebler \cite{CL}
as classes of lines in $\PG(3,q)$ satisfying certain geometric conditions.
While their equivalence was observed by Drudge \cite{D} via the Klein-correspondence,
they have been studied for almost three decades under the name of Cameron-Liebler line classes \cite{P91}.

Further,
Drudge \cite[Corollary~9.1]{D} proved that,
for $q\geq 4$, any $x$-tight set of $\mathsf{Q}^+(2n+1,q)$ with $x\leq \sqrt{q}$ is the disjoint union
of $x$ generators (hence, none such exist if $x>2$ and the rank $n+1$ is odd).
This bound was improved several times \cite{tightpolar,DGHS} and finally
Beukemann and Metsch \cite{BM} proved that the same conclusion holds
if $1\leq n\leq 3$ and $x\leq q$, or if $n\geq 4$, $q\geq 71$ and $x\leq q-1$.

A much stronger bound is shown \cite{M} for the hyperbolic polar space $\mathsf{Q}^+(5,q)$ of rank $3$:
the parameter $x$ of a tight set that is not a disjoint union of generators
should satisfy $x\geq cq^{4/3}$ with some constant $c$.
Moreover, the following result was obtained in \cite{GM}.


\begin{res}\label{main1}
Let $\mathcal{T}$ be a tight set with parameter $x$ of $\mathsf{Q}^+(5,q)$.
Then, for every generator $G$ of $\mathsf{Q}^+(5,q)$, the number $w:=|G\cap \mathcal{T}|$
satisfies
\begin{equation}\label{eqn_main}
{x\choose 2}+w(w-x)\equiv 0\mod q+1.
\end{equation}
\end{res}

Thus, if $\mathsf{Q}^+(5,q)$ has a tight set with parameter $x$,
then Eq. (\ref{eqn_main}) has a solution in $w$ from the set $\{0,1,\dots,q\}$. It was shown
in \cite{GM} that it implies a strong non-existence result for tight sets of $\mathsf{Q}^+(5,q)$,
as, for any given $q$, it rules out roughly at least half of the possible parameters $x$
from the set $\{3,\ldots,\frac{q^2+1}{2}\}$.

The main result of the present paper generalizes Result \ref{main1}
to the hyperbolic polar space $\mathsf{Q}^+(2n+1,q)$ of an arbitrary rank $n+1$ as follows.

\begin{theorem}\label{theo-main}
Let $\mathcal{T}$ be a tight set with parameter $x$ of $\mathsf{Q}^+(2n+1,q)$.
Then, for every generator $G$ of $\mathsf{Q}^+(2n+1,q)$, the number $w:=|G\cap \mathcal{T}|$ satisfies
\begin{equation}\label{eqn_main-even}
{x\choose 2}+w(w-x)\equiv 0\mod q+1,
\end{equation}
if $n$ is even, and
\begin{equation}\label{eqn_main-odd}
w(w-x)\equiv 0\mod q+1,
\end{equation}
if $n$ is odd.
\end{theorem}

The argument from \cite[Section~3]{GM} shows that Eq. (\ref{eqn_main-even})
rules out roughly at least half of the possible parameters $x$ from the set
$\{3,\ldots,\frac{q^{n}+1}{2}\}$. Unfortunately, if $n$ is odd, it seems that
Eq. (\ref{eqn_main-odd}) does not impose any restrictions on $x$, which is consistent
with the existence of large sets of disjoint generators in $\mathsf{Q}^+(2n+1,q)$ in this case.

The proof of Theorem \ref{theo-main} generalizes a technique developed in \cite{GM,GM14}.
Suppose that $\mathcal{T}$ is a tight set with parameter $x$ of $\mathsf{Q}:=\mathsf{Q}^+(2n+1,q)$.
It follows \cite[Theorem~8.1]{D} that, for any point $P$ of $\mathsf{Q}$, one has
\begin{equation}\label{eq-tightset-point}
|P^{\perp}\cap \mathcal{T}|=q^n|\{P\}\cap \mathcal{T}|+x\theta_{n-1},
\end{equation}
where, for an integer $k\geq -1$, we define $\theta_k:=(q^{k+1}-1)/(q-1)$. 
Observe that $\theta_k$ is the number of points in the $k$-dimensional projective space
over $\mathbb{F}_q$, and $|\mathcal{T}|=x\theta_n$.

The following result \cite[Lemma~2.1]{BM} generalizes Eq. (\ref{eq-tightset-point})
to the subspaces of $\PG(2n+1,q)$, and it will play a crucial role in the proof.

\begin{res}\label{res-1}
Every $s$-dimensional subspace $S$, $s\leq n$, of the ambient space $\PG(2n+1,q)$ satisfies
the equality
\[
|S^{\perp}\cap \T|=q^{n-s}|S\cap \T|+x\theta_{n-s-1}.
\]
\end{res}

Drudge \cite[Theorem~8.1(3)]{D} proved a partial case of Result \ref{res-1}
when $S$ is a line not contained in $\mathsf{Q}$ (i.e., $S$ intersects $\mathsf{Q}$
in two non-collinear points).
For a point $P_0\in \mathsf{Q}\setminus \mathcal{T}$, applying Result \ref{res-1} in different settings,
we count the number of pairs $(P_1,P_2)$ of collinear points $P_1\in P_0^{\perp}\cap \mathcal{T}$,
$P_2\in (\mathsf{Q}\setminus P_0^{\perp})\cap \mathcal{T}$ in two ways,
which gives a certain equation. We then analyze this equation modulo $q+1$.

Note that via the field reduction \cite{K} one can construct
tight sets of $\mathsf{Q}^+(6b-1,q)$ from a variety of those of $\mathsf{Q}^+(5,q^b)$
that have been discovered recently \cite{CP17,CP18,CP19,DDMR,FMX,GMP18}.
Thus, $\mathsf{Q}^+(7,q)$ and $\mathsf{Q}^+(9,q)$ seem to be the first unexplored cases,
where we are not aware of any non-trivial tight sets.


\section{The proof of Theorem \ref{theo-main}}

We first recall some well-known properties of a hyperbolic polar space \cite{pps}.

\begin{res}\label{res-pps}
Let $\mathsf{Q}$ be a hyperbolic polar space $\Q(2n+1,q)$ or rank $n+1$.
\begin{enumerate}
\renewcommand{\labelenumi}{\rm (\alph{enumi})}
\item For every pair $P_1,P_2$ of non-collinear points of $\mathsf{Q}$,
$\{P_1,P_2\}^{\perp}$ is a hyperbolic quadric $\Q(2n-1,q)$ of rank $n$.
\item For every $s$-subspace $S\subset \mathsf{Q}$, the quotient space $S^{\perp}/S$
is a hyperbolic quadric of rank $n-s$ (over the same field).
\item The number of points of $\mathsf{Q}$ is $k_n:=(q^n+1)\theta_n$,
of which $k_{2,n}:=q^{2n}$ are not collinear to a given point.
\item For every pair of distinct collinear points of $\mathsf{Q}$,
the number of points of $\mathsf{Q}$ that are collinear to only one of them equals
$b_{n}:=q^{2n-1}$.
\end{enumerate}
\end{res}


\begin{lemma}\label{lemma-lm}
Let $\mathsf{Q}$ be a hyperbolic polar space $\Q(2n-1,q)$ or rank $n$.
For a point $P_0$ of $\mathsf{Q}$,
let $P$ and $P'$ be two distinct points of $P_0^{\perp}\setminus \{P_0\}$,
and $\ell:=\langle P,P'\rangle$. 
\begin{enumerate}
\renewcommand{\labelenumi}{\rm (\alph{enumi})}
\item If $\ell\subset \mathsf{Q}$ with $P_0\in \ell$, then
$P^{\perp}\cap P'^{\perp}\subseteq P_0^{\perp}$.
\item If $\ell\subset \mathsf{Q}$ with $P_0\notin \ell$,
then $P^{\perp}\cap P'^{\perp}$ contains precisely $\lambda_n:=q^{2n-4}$ points
of $\mathsf{Q}\setminus P_0^{\perp}$.
\item If $\ell\not\subset \mathsf{Q}$,
then $P^{\perp}\cap P'^{\perp}$ contains precisely $\lambda_n$ points
of $\mathsf{Q}\setminus P_0^{\perp}$.
\end{enumerate}
\end{lemma}
\begin{proof}
Statement (a) is obvious.
To prove (b), we observe that $\ell^{\perp}/\ell$
is a hyperbolic quadric of rank $n-2$, 
whose points are the planes on $\ell$, 
and therefore
the number of points in $P^{\perp}\cap P'^{\perp}$
equals $\theta_1+k_{n-3}(\theta_2-\theta_1)$, of which
$\theta_2$ points are contained in the plane
$\langle P_0,\ell\rangle$ and
$k_{n-4}(\theta_3-\theta_2)$ points correspond to 
$\langle P_0,\ell\rangle^{\perp}/\langle P_0,\ell\rangle$.
Thus, the number of points in $(P^{\perp}\cap P'^{\perp})\setminus P_0^{\perp}$
equals:
\[
\theta_1+k_{n-3}(\theta_2-\theta_1) - \theta_2 - k_{n-4}(\theta_3-\theta_2)=q^{2n-4},
\]
which shows (b). The proof of (c) is similar.
\end{proof}

We will need the following technical lemma.
For a point $P$ of a hyperbolic quadric, define $\mathsf{L}(P)$
to be the set of lines on the quadric through $P$.

\begin{lemma}\label{lemma-mu}
Let $\mathsf{Q}$ be a hyperbolic quadric $\Q(2n-1,q)$ of rank $n$.
Suppose that $\mu$ is 
an integer-valued function defined
on the set of points of $\mathsf{Q}$ such that, for a positive integer $x$,
the following properties hold.
\begin{enumerate}
\renewcommand{\labelenumi}{\rm (\alph{enumi})}
\item[$(*)$] For every point $P$ of $\mathsf{Q}$:
\[
\sum_{P_1\in P^{\perp}}\mu(P_1) = x\theta_{n-2}+q^{n-1}\mu(P).
\]
\item[$(**)$] For every pair $P_1,P_2$ of non-collinear points of $\mathsf{Q}$:
\[
\sum_{P'\in \{P_1,P_2\}^{\perp}}\mu(P') = x\theta_{n-3}+q^{n-2}(\mu(P_1)+\mu(P_2)).
\]
\end{enumerate}
Then, for an arbitrary point $P_0$ of $\mathsf{Q}$,
one has:
\[
\sum_{P\in \mathsf{Q}}\mu(P)^2 =
\mu(P_0)^2 + (x-\mu(P_0))^2 + (q+1)\cdot\sum_{P_1\in P_0^{\perp}\setminus \{P_0\}}\mu(P_1)^2
-\sum_{\ell\in \mathsf{L}(P_0)}\big(\sum_{P_2\in \ell\setminus \{P_0\}}\mu(P_2)\big)^2.
\]
\end{lemma}
\begin{proof}
For a point $P_0\in \mathsf{Q}$, one can write:
\begin{equation}\label{eq-dist-sum}
\sum_{P\in \mathsf{Q}} \mu(P)^2 = \mu(P_0)^2 +
\sum_{P_1\in P_0^{\perp}\setminus \{P_0\}}\mu(P_1)^2 +
\sum_{P_2\in \mathsf{Q}\setminus P_0^{\perp}}\mu(P_2)^2.
\end{equation}

For a point $P_2\in \mathsf{Q}\setminus P_0^{\perp}$,
Property $(**)$ implies that:
\[
\mu(P_2)=
\frac{1}{q^{n-2}}\Big(\sum_{P'\in \{P_0,P_2\}^{\perp}}\mu(P')-q^{n-2}\mu(P_0)-x\theta_{n-3}\Big),
\]
which allows to rewrite the last sum in Eq. (\ref{eq-dist-sum}) as follows:
\begin{align*}
\sum_{P_2\in \mathsf{Q}\setminus P_0^{\perp}}\mu(P_2)^2 &=
\frac{1}{q^{2n-4}}\sum_{P_2\in \mathsf{Q}\setminus P_0^{\perp}}
\Big(\sum_{P'\in \{P_0,P_2\}^{\perp}}\mu(P')-q^{n-2}\mu(P_0)-x\theta_{n-3}\Big)^2\\
&=\frac{1}{q^{2n-4}}\sum_{P_2\in \mathsf{Q}\setminus P_0^{\perp}}
\Bigg(
\Big(\sum_{P'\in \{P_0,P_2\}^{\perp}}\mu(P')\Big)^2
-
2\Big(\sum_{P'\in \{P_0,P_2\}^{\perp}}\mu(P')\Big)\big(q^{n-2}\mu(P_0)+x\theta_{n-3}\big)\\
&+
\big(q^{n-2}\mu(P_0)+x\theta_{n-3}\big)^2
\Bigg).
\end{align*}

Further, we observe that
\begin{align*}
\sum_{P_2\in \mathsf{Q}\setminus P_0^{\perp}}\sum_{P'\in \{P_0,P_2\}^{\perp}}\mu(P')
&= b_{n-1}\cdot \sum_{P_1\in P_0^{\perp}\setminus \{P_0\}}\mu(P_1)
\mbox{~\big[holds by Result \ref{res-pps}(d)\big]}\\
&= q^{2n-3}\big((q^{n-1}-1)\mu(P_0)+x\theta_{n-2}\big)
\mbox{~\big[holds by Property $(*)$\big]},\\
\mbox{~and~}\sum_{P_2\in \mathsf{Q}\setminus P_0^{\perp}}1
&= k_{2,n-1}
\mbox{~\big[holds by Result \ref{res-pps}(b),(c)\big]}\\
&=q^{2n-2}.
\end{align*}

Thus, we obtain:
\begin{align*}
\sum_{P_2\in \mathsf{Q}\setminus P_0^{\perp}}\mu(P_2)^2 &=
\frac{1}{q^{2n-4}}\sum_{P_2\in \mathsf{Q}\setminus P_0^{\perp}}
\Big(\sum_{P'\in \{P_0,P_2\}^{\perp}}\mu(P')\Big)^2+
q^2\big(q^{n-2}\mu(P_0)+x\theta_{n-3}\big)^2\\
& -2q\big(q^{n-2}\mu(P_0)+x\theta_{n-3}\big)\big((q^{n-1}-1)\mu(P_0)+x\theta_{n-2}\big),
\end{align*}
where we shall evaluate the first double sum by using Lemma \ref{lemma-lm}.
Indeed, for any pair $P,P'$ of points of $P_0^{\perp}\setminus \{P_0\}$
such that $P_0\notin\langle P,P'\rangle$, there are precisely $\lambda_n$
points $P_2\in \mathsf{Q}\setminus P_0^{\perp}$ with $P,P'\in \{P_0,P_2\}^{\perp}$.
Thus, it follows that
\begin{align*}
\sum_{P_2\in \mathsf{Q}\setminus P_0^{\perp}}
\big(\sum_{P'\in \{P_0,P_2\}^{\perp}}\mu(P')\big)^2 &=
b_{n-1}\cdot \sum_{P_1\in P_0^{\perp}\setminus \{P_0\}}\mu(P_1)^2 + \lambda_n\cdot
\sum_{\ell\in \mathsf{L}(P_0)}\sum_{P\in \ell\setminus \{P_0\}}\mu(P)
\sum_{P'\in P_0^{\perp}\setminus \ell}\mu(P'),
\end{align*}
where the last triple sum can be rewritten as follows:
\begin{align*}
\sum_{\ell\in \mathsf{L}(P_0)}\sum_{P\in \ell\setminus \{P_0\}}\mu(P)
\sum_{P'\in P_0^{\perp}\setminus \ell}\mu(P') &=
\sum_{\ell\in \mathsf{L}(P_0)}\sum_{P\in \ell\setminus \{P_0\}}\mu(P)
\Big(
\sum_{P'\in P_0^{\perp}\setminus \{P_0\}}\mu(P')-
\sum_{P''\in \ell\setminus \{P_0\}}\mu(P'')
\Big)\\
&\Big[
\mbox{by~}\sum_{\ell\in \mathsf{L}(P_0)}\sum_{P\in \ell\setminus \{P_0\}}\mu(P)=
\sum_{P_1\in P_0^{\perp}\setminus \{P_0\}}\mu(P_1)
\Big]\\
&= \Big(\sum_{P_1\in P_0^{\perp}\setminus \{P_0\}}\mu(P_1)\Big)^2
-\sum_{\ell\in \mathsf{L}(P_0)}\big(\sum_{P\in \ell\setminus \{P_0\}}\mu(P)\big)^2\\
&= \big((q^{n-1}-1)\mu(P_0)+x\theta_{n-2}\big)^2
-\sum_{\ell\in \mathsf{L}(P_0)}\big(\sum_{P\in \ell\setminus \{P_0\}}\mu(P)\big)^2.
\end{align*}

Putting it all together and simplifying, we obtain:
\begin{align*}
\sum_{P\in \mathsf{Q}}\mu(P)^2  &=
\mu(P_0)^2 +
\sum_{P_1\in P_0^{\perp}\setminus \{P_0\}}\mu(P_1)^2 +
\sum_{P_2\in \mathsf{Q}\setminus P_0^{\perp}}\mu(P_2)^2 \\
&=
\mu(P_0)^2 +
(q+1)\cdot\sum_{P_1\in P_0^{\perp}\setminus \{P_0\}}\mu(P_1)^2
-\sum_{\ell\in \mathsf{L}(P_0)}\big(\sum_{P\in \ell\setminus \{P_0\}}\mu(P)\big)^2\\
&
+\big((q^{n-1}-1)\mu(P_0)+x\theta_{n-2}\big)^2
+
q^2\big(q^{n-2}\mu(P_0)+x\theta_{n-3}\big)^2\\
&-
2q\big((q^{n-1}-1)\mu(P_0)+x\theta_{n-2}\big)\big(q^{n-2}\mu(P_0)+x\theta_{n-3}\big)
\\
&=
\mu(P_0)^2 + \big(x-\mu(P_0)\big)^2
+(q+1)\sum_{P_1\in P_0^{\perp}\setminus \{P_0\}}\mu(P_1)^2
-\sum_{\ell\in \mathsf{L}(P_0)}\big(\sum_{P\in \ell\setminus \{P_0\}}\mu(P)\big)^2,
\end{align*}
and the lemma follows.
\end{proof}

One can see that a function $\mu$ satisfying the condition of Lemma \ref{lemma-mu}
generalizes the notion of tight sets.
The proof of the following lemma justifies this by showing that $\mu$ is a weighted
tight set \cite{BBI} and, moreover, given such a function $\mu$ one can construct
a weighted set for hyperbolic quadrics of smaller rank.

\begin{lemma}\label{lemma-weight}
Let $\mathsf{Q}$ be a hyperbolic quadric $\Q(2n-1,q)$ of rank $n$.
Suppose that $\mu$ is 
an integer-valued function defined
on the set of points of $\mathsf{Q}$ such that, for a positive integer $x$,
$\mu$ satisfies Property $(*)$ of Lemma \ref{lemma-mu}.
Then the following holds.
\begin{enumerate}
\renewcommand{\labelenumi}{\rm (\alph{enumi})}
\item[$(1)$] $\mu$ satisfies Property $(**)$ of Lemma \ref{lemma-mu}.
\item[$(2)$] For a point $P_0\in \mathsf{Q}$,
a function $\widetilde{\mu}$ on the points of the hyperbolic quadric $P_0^{\perp}/P_0$ defined by:
\[
\widetilde{\mu}(\ell)=\sum_{P\in \ell\setminus\{P_0\}}\mu(P),~~(\ell\in \mathsf{L}(P_0)),
\]
satisfies Property $(*)$ with $\widetilde{x}=(q-1)\mu(P_0)+x$.
\end{enumerate}
\end{lemma}
\begin{proof}
First observe that by Property $(*)$ and Result \ref{res-pps}(b),(c), we have
\begin{eqnarray*}
  \sum_{P\in \mathsf{Q}}\mu(P) &=& \frac{1}{q^{n-1}}\Big(\sum_{P\in \mathsf{Q}}\sum_{P_1\in P^{\perp}}\mu(P_1)\Big)-k_{n-1}\frac{x\theta_{n-2}}{q^{n-1}} \\
  &=& \frac{1+qk_{n-2}}{q^{n-1}}\Big(\sum_{P\in \mathsf{Q}}\mu(P)\Big)-k_{n-1}\frac{x\theta_{n-2}}{q^{n-1}},
\end{eqnarray*}
which implies that $M:=\sum_{P\in \mathsf{Q}}\mu(P)=x\theta_{n-1}$.

$(1)$ Recall that the collinearity graph of $\mathsf{Q}$, in which
two distinct points are adjacent whenever they are collinear, is strongly regular.
Therefore its symmetric $(0,1)$-adjacency matrix is diagonalizable with three distinct eigenvalues,
namely, $\sigma_0=qk_{n-2}$ with multiplicify one (and the all-one eigenvector),
$\sigma_1=q^{n-1}-1$ and $\sigma_2=-(q^{n-2}+1)$, see \cite[Lemma~8.3]{D}.

It now follows from Property $(*)$ that a vector $\mathbf{v}\in \mathbb{R}^{\mathsf{Q}}$
defined by
$\mathbf{v}(P)=\mu(P)+\frac{x\theta_{n-2}}{q^{n-1}-1-qk_{n-2}}$
is an eigenvector for $\sigma_1$, as it satisfies
\[
\mathbf{v}(P)=(q^{n-1}-1)\cdot\sum_{P_1\in P^{\perp}\setminus\{P\}}\mathbf{v}(P_1)\text{~for~any~}P\in \mathsf{Q}.
\]

Given two non-collinear points $P,P'$ of $\mathsf{Q}$, it follows from the proof of \cite[Theorem~8.1(3)]{D}
that
\[
\mathbf{w}_{P,P'}=\mathbf{e}_{\{P,P'\}^{\perp}}-\frac{\theta_{n-3}}{\theta_{n-1}}\mathbf{e}_{\mathsf{Q}}-q^{n-2}(\mathbf{e}_{\{P\}}+\mathbf{e}_{\{P'\}}),
\]
where $\mathbf{e}_X\in \mathbb{R}^{\mathsf{Q}}$ denotes the characteristic vector of a point set $X$,
is an eigenvector for $\sigma_2$.

Since the eigenspaces associated to $\sigma_1$ and $\sigma_2$ are orthogonal,
it follows that $\mathbf{v}$ is perpendicular with $\mathbf{w}_{P,P'}$. 
Evaluating $\langle \mathbf{v},\mathbf{w}_{P,P'}\rangle=0$ and simplifying, 
we find that Property $(**)$ follows.

$(2)$ We follow the proof of \cite[Lemma~2.1]{BM}.
Fix a line $\ell_0\in \mathsf{L}(P_0)$.
If we consider the subspaces $P^{\perp}$ for the $q+1$ points $P$ of $\ell_0$, then
the points of $\ell_0^{\perp}$ lie in each such subspace $P^{\perp}$, whereas
every point outside $\ell_0^{\perp}$ lies in exactly one of these.
By double counting, we find:
\begin{eqnarray*}
  q\Big(\sum_{P\in\ell_0^{\perp}}\mu(P)\Big) + M &=& \sum_{P_1\in \ell_0}\sum_{P_2\in P_1^{\perp}}\mu(P_2) \\
   &=& \sum_{P_1\in \ell_0}\big(q^{n-1}\mu(P_1)+x\theta_{n-2}\big) \\
   &=& q^{n-1}\widetilde{\mu}(\ell_0)+q^{n-1}\mu(P_0)+(q+1)x\theta_{n-2},
\end{eqnarray*}
where on the left-hand side we have:
\[
\sum_{P\in\ell_0^{\perp}}\mu(P)=\mu(P_0)+\widetilde{\mu}(\ell_0)+\sum_{\ell_1\in (\mathsf{L}(P_0)\setminus\{\ell_0\})\cap \ell_0^{\perp}}\widetilde{\mu}(\ell_1),
\]
and simplifying gives the desired property $(*)$ for $\widetilde{\mu}$.
\end{proof}

For the rest of the proof, we assume that
$\mathcal{T}$ is a tight set with parameter $x$
of a hyperbolic quadric $\mathsf{Q}=\Q(2n+1,q)$ of rank $n+1$.
Fix a point $P_0\in \mathsf{Q}\setminus \mathcal{T}$
and, for every line $\ell$ of $\mathsf{L}(P_0)$, put $m_{\ell}:=|\ell\cap \T|$.

\begin{lemma}\label{lemma-sum2}
The following holds:
\begin{equation}\label{eq-sum2}
\sum_{\ell\in \mathsf{L}(P_0)} m_{\ell}^2 = x(\theta_{n-1}-1 + x).
\end{equation}
\end{lemma}
\begin{proof}
We prove the result by double counting the number $E$ of pairs $(P_1,P_2)$
where $P_1\in P_0^{\perp}\cap \T$, $P_2\in \T\setminus P_0^{\perp}$ and
$P_1\in P_2^{\perp}$.

Observe that, by Eq. (\ref{eq-tightset-point}), there are precisely
$|\T\setminus P_0^{\perp}|=|\T|-|P_0^{\perp}\cap \T|=xq^n$ points $P_2$ of $\T$ that are not collinear to $P_0$.
By Result \ref{res-1} applied to a line $\langle P_0,P_2\rangle\not\subset \mathsf{Q}$, $P_2\in \mathcal{T}$,
each of them is collinear to
\[
|\langle P_0,P_2\rangle^{\perp}\cap \T|=q^{n-1}|\langle P_0,P_2\rangle\cap \T|+x\theta_{n-2}
=q^{n-1}+x\theta_{n-2}
\]
points $P_1\in P_0^{\perp}\cap \T$, as $\langle P_0,P_2\rangle\cap \T=\{P_2\}$.
Thus, $E$ equals $xq^n(q^{n-1}+x\theta_{n-2})$.

On the other hand, for a point $P_1\in P_0^{\perp}\cap \T$, the number of points
$P_2\in \T\setminus P_0^{\perp}$ that are collinear to $P_1$ equals
\[
|(P_1^{\perp}\setminus P_0^{\perp})\cap \T|=|P_1^{\perp}\cap \T|-|\langle P_0,P_1\rangle^{\perp}\cap \T|=
q^n+q^{n-1}(x - m_{\ell})
\]
where we apply Result \ref{res-1} to the line $\ell:=\langle P_0,P_1\rangle$,
and observe that this number does not depend on the particular choice of a point $P_1$ of $\ell \cap \T$.
Therefore, $E$ equals
\[
\displaystyle{\sum_{\ell\in \mathsf{L}(P_0)}}m_{\ell}(q^n+q^{n-1}(x - m_{\ell})),
\]
and note that by Eq. (\ref{eq-tightset-point}) and $P_0\notin \T$, we have:
\begin{eqnarray*}
  |P_0^{\perp}\cap \T| &=& \sum_{\ell\in \mathsf{L}(P_0)}m_{\ell} \\
   &=& x\theta_{n-1}.
\end{eqnarray*}

Thus, we obtain that
\[
\displaystyle{\sum_{\ell\in \mathsf{L}(P_0)}}m_{\ell}(q^n+q^{n-1}(x - m_{\ell}))=xq^n(q^{n-1}+x\theta_{n-2}),
\]
which simplifies to
\[
\sum_{\ell\in \mathsf{L}(P_0)} m_{\ell}^2 = x(\theta_{n-1}-1 + x),
\]
and the lemma follows.
\end{proof}

In the following lemma, using Lemma \ref{lemma-mu},
we obtain another result for the left-hand side of Eq. (\ref{eq-sum2}).
Let $c\in \{0,1\}$ be defined such that $n\equiv c \pmod 2$.

\begin{lemma}\label{lemma-congr}
For any generator $G$ of $\mathsf{Q}$ on $P_0$ with $w:=|G\cap \mathcal{T}|$,
the following equality holds:
\[
\sum_{\ell\in \mathsf{L}(P_0)}m_{\ell}^2 \equiv x(\theta_{n-1}-c)+(-1)^c\cdot 2w(x-w)+cx^2
 \mod 2(q+1).
\]
\end{lemma}
\begin{proof}
Recall that $P_0^{\perp}/P_0$ is a hyperbolic quadric of rank $n$,
whose points are the lines of $\mathsf{L}(P_0)$.
Let $\mu$ be a function defined on $\mathsf{L}(P_0)$ such that
$\mu$ satisfies the condition of Lemma \ref{lemma-mu}.
By Lemma \ref{lemma-mu}, for an arbitrary line $\ell_0\in\mathsf{L}(P_0)$
and the set $\mathsf{P}(\ell_0)$ of planes of $\mathsf{Q}$ on the line $\ell_0$,
we obtain that
\[
\sum_{\ell\in\mathsf{L}(P_0)}\mu(\ell)^2 =
\mu(\ell_0)^2 +\big(x-\mu(\ell_0)\big)^2 
+ (q+1)\sum_{\ell_1\in \mathsf{L}(P_0)\cap (\ell_0^{\perp}\setminus \{\ell_0\})}\mu(\ell_1)^2
-\sum_{\pi\in \mathsf{P}(\ell_0)}\big(\sum_{\ell'\in \mathsf{L}(P_0)\cap (\pi\setminus \{\ell_0\})}\mu(\ell')\big)^2,
\]
which is congruent modulo $2(q+1)$ to
\begin{equation}\label{eq-induction}
\mu(\ell_0)^2 +\big(x-\mu(\ell_0)\big)^2
+ (q+1)\big((q^{n-1}-1)\mu(\ell_0)+x\theta_{n-2}\big)
-\sum_{\pi\in \mathsf{P}(\ell_0)}\big(\sum_{\ell'\in \mathsf{L}(P_0)\cap (\pi\setminus \{\ell_0\})}\mu(\ell')\big)^2, 
\end{equation}
as modulo $2$ we have $m^2\equiv m$ for every integer $m$ and
$\sum_{\ell_1\in \mathsf{L}(P_0)\cap (\ell_0^{\perp}\setminus \{\ell_0\})}\mu(\ell_1)=(q^{n-1}-1)\mu(\ell_0)+x\theta_{n-2}$
by Property $(*)$.

Further, consider a hyperbolic quadric $\ell_0^{\perp}/\ell_0$, which has rank $n-1$ and whose
points are the planes of $\mathsf{P}(\ell_0)$,
and, as in Lemma \ref{lemma-weight}(2), define a function $\widetilde{\mu}$ by
\[ 
\widetilde{\mu}(\pi):=\sum_{\ell'\in \mathsf{L}(P_0)\cap (\pi\setminus \{\ell_0\})}\mu(\ell')
\]
on the set $\mathsf{P}(\ell_0)$.
By Lemma \ref{lemma-weight},
$\widetilde{\mu}$ satisfies the condition of Lemma \ref{lemma-mu}
for the hyperbolic quadric $\ell_0^{\perp}/\ell_0$,
and therefore the last double sum in Eq. (\ref{eq-induction}) can be evaluated by induction on $n$.

We now define $\mu$ by $\mu(\ell):=m_{\ell}$ for a line $\ell\in\mathsf{L}(P_0)$.
By Result \ref{res-1}, $\mu$ satisfies the condition of Lemma \ref{lemma-mu}.
For $n+1=3$, it follows from the proof of \cite[Theorem~3.1]{GM} that
\[
\sum_{\ell\in\mathsf{L}(P_0)}\mu(\ell)^2 \equiv
x(q+1)+2w(x-w)
\mod 2(q+1),
\]
where $w=\sum_{\ell\in \mathsf{L}(P_0)\cap \pi}\mu(\ell)$, i.e.,
the number of points of $\T$ in any plane $\pi$ on $P_0$.

Therefore, for $n+1=4$, from Eq. (\ref{eq-induction}) we obtain that
\begin{align*}
\sum_{\ell\in\mathsf{L}(P_0)}m_{\ell}^2 &\equiv
m_{\ell_0}^2 +\big(x-m_{\ell_0}\big)^2
+ (q+1)((q^{2}-1)m_{\ell_0}+x\theta_{1})
-\sum_{\pi\in \mathsf{P}(\ell_0)}\widetilde{\mu}(\pi)^2 \\
&\equiv
m_{\ell_0}^2 +\big(x-m_{\ell_0}\big)^2
+ (q+1)((q^{2}-1)m_{\ell_0}+x\theta_{1})
-\big(\widetilde{x}(q+1)+2\widetilde{w}(\widetilde{x}-\widetilde{w})\big)
\mod 2(q+1),
\end{align*}
where $\widetilde{w}:=\sum_{\pi\in \mathsf{P}(\ell_0)\cap G}\widetilde{\mu}(\pi)=|(G\setminus \{\ell_0\})\cap \T|$
so that $w=\widetilde{w}+m_{\ell_0}$
for any 3-dimensional space $G$ of $\mathsf{Q}$ on $\ell_0$,
and this simplifies to
\[
\sum_{\ell\in\mathsf{L}(P_0)}m_{\ell}^2 \equiv
x(q^2+q+x)
-2w(x-w)
\mod 2(q+1).
\]

Arguing in the same manner for $n+1>4$ by induction, a routine (but tedious) check shows that
\[
\sum_{\ell\in \mathsf{L}(P_0)}m_{\ell}^2 \equiv x(\theta_{n-1}-1+x)-2w(x-w)
 \mod 2(q+1).
\]
if $n$ is odd, and
\[
\sum_{\ell\in \mathsf{L}(P_0)}m_{\ell}^2 \equiv x\theta_{n-1}+2w(x-w)
 \mod 2(q+1).
\]
if $n$ is even,
which shows the lemma.
\end{proof}

We are now in a position to prove our main result.
By Lemmas \ref{lemma-sum2} and \ref{lemma-congr},
we obtain that, if $n$ is odd then
\begin{eqnarray*}
  \sum_{\ell\in \mathsf{L}(P_0)} m_{\ell}^2 &\equiv & x(\theta_{n-1}-1 + x) \\
   &\equiv& x(\theta_{n-1}-1+x)-2w(x-w)
 \mod 2(q+1),
\end{eqnarray*}
i.e.,
\[
  w(w-x) \equiv 0 \mod (q+1),
\]
and if $n$ is even then
\begin{eqnarray*}
  \sum_{\ell\in \mathsf{L}(P_0)} m_{\ell}^2 &\equiv & x(\theta_{n-1}-1 + x) \\
   &\equiv& x\theta_{n-1}+2w(x-w)
 \mod 2(q+1),
\end{eqnarray*}
i.e.,
\[
  {x \choose 2}+w(w-x) \equiv 0 \mod (q+1),
\]
which completes the proof of Theorem \ref{theo-main}.

\Acknowledgements

\bibliographystyle{abbrv}
\bibliography{references}

\end{document}